\documentclass[smallextended,referee,envcountsect]{arxiv-jota} \smartqed\usepackage{graphicx} \usepackage{amsmath} \usepackage{amssymb} \usepackage{mathtools} \usepackage{amsfonts} \usepackage{cite} \journalname{JOTA} \par
\let\arxivOriginalHeadbox\makeheadbox\renewcommand{\makeheadbox}{\phantom{\arxivOriginalHeadbox}} \hbadness=10000\vbadness=10000\hfuzz=3.2pt\begin{document} \par
\title{Sharp Quadratic Majorants of Weighted Power Means of Quadratic Forms} \par
\author{Alexey Peregudin} \par
\institute{Alexey Peregudin \at University of Sheffield \\ Sheffield, United Kingdom\\ a.peregudin@sheffield.ac.uk } \par
\date{Received: date / Accepted: date} \par
\maketitle\par
\begin{abstract} We study when concave weighted power means of positive semidefinite quadratic forms admit sharp quadratic upper bounds. The starting point is the scalar fact that each such mean is the lower envelope of its supporting affine majorants. After the scalar variables are replaced by quadratic forms, the resulting bounds may cease to be extremal among quadratic majorants. We prove that sharpness is preserved under a lower-envelope convexity condition on the joint numerical range. In that case, a nonlinear inequality involving power means of quadratic forms is certified, with no gap, by a single linear matrix inequality. The condition is weaker than convexity of the full augmented range and holds automatically for two quadratic forms. This yields a sharp Peter-Paul-type majorant for the whole concave power-mean scale, including Yuan's lemma and the weighted geometric mean as special cases. We also derive variational formulas for the extrema of sums of power means of quadratic forms on the unit sphere; under the same lower-envelope condition, the maximization becomes an exact convex eigenvalue minimization. A finite-horizon control application illustrates how the certificate can be used constructively. \end{abstract}  \par
\section{Introduction} \label{sec:intro} \par
The family of weighted power means provides a unified way of averaging nonnegative values, interpolating between the minimum and maximum operators. We begin by recalling the standard definition of these means. \par
Consider a vector $y \in\mathbb{R}^m_{\scriptscriptstyle\ge0}$ and weights $w \in\mathbb{R}^m_{\scriptscriptstyle> 0}$ satisfying $\sum_{i=1}^m w_i = 1$. For $p \in\mathbb{R} \setminus\{0\}$, the weighted power mean of $y$ with weights $w$ is defined as \begin{equation} \label{eq:Mean-definition} \underset{p, \;w}{\operatorname{Mean}}\big[y\big] \vcentcolon= \left( \sum_{i=1}^m w_i y_i^p \right)^{1/p} \end{equation} with the convention that $\operatorname{Mean}_{p,w}[y]=0$ if $p < 0$ and some $y_i=0$. By continuous extension, the definition includes the limiting cases $p \in\{0, \pm\infty\}$: \begin{equation*} \underset{-\infty, \;w}{\operatorname{Mean}} \big[y\big] = \min_{1 \le i \le m} y_i, \qquad\underset{0, \;w}{\operatorname{Mean}} \big[y\big] = \prod_{i=1}^m y_i^{w_i}, \qquad\underset{+\infty, \;w}{\operatorname{Mean}} \big[y\big] = \max_{1 \le i \le m} y_i. \end{equation*} This definition generalizes standard weighted means. For instance, setting the power $p$ to $-1, 0, 1, 2$ yields the harmonic, geometric, arithmetic, and quadratic means, respectively. \par
\par
For a fixed $w$ and $p \in[-\infty,1)$, let $q = p(p-1)^{-1}$, with the limiting convention $q=1$ when $p=-\infty$, and define the admissible coefficient set $\mathcal{A}$ as the level set of the conjugate mean: \begin{equation*} \mathcal{A} = \left\{ \alpha\in\mathbb{R}_{\scriptscriptstyle\ge0}^{m} \;\middle|\; \underset{q, \;w}{\operatorname{Mean}}\big[\alpha\big] = 1 \right\}. \end{equation*} Then, for any $\alpha\in\mathcal{A}$, it holds that $\operatorname{Mean}_{p,w} [y] \le\sum_{i=1}^m \alpha_i w_i y_i.$ Moreover, this inequality is \textit{sharp}: \begin{equation} \label{eq:Var-Mean} \underset{p, \;w}{\operatorname{Mean}} \big[y\big] = \inf_{\alpha\in\mathcal{A}} \sum_{i=1}^m \alpha_i w_i y_i. \end{equation} \par
The identity \eqref{eq:Var-Mean} is the standard supporting-hyperplane representation of a concave weighted power mean; see, for example, Bullen's handbook of means~\cite{bullen2003handbook}. Motivated in part by applications in control theory, the present paper asks when this scalar representation remains sharp after replacing the scalar variables by positive semidefinite quadratic forms. Specifically, let $A_i \in\mathbb{S}^n_{\scriptscriptstyle\succeq0}$. Substituting $y_i = x^\top A_i x$ into \eqref{eq:Var-Mean}, we observe that for all $x\in\mathbb{R}^n$, \begin{equation} \label{eq:PowerMean-Quad} \underset{p, \;w}{\operatorname{Mean}} \big[x^\top A_1 x, \dots, x^\top A_m x\big] \le\sum_{i=1}^m \alpha_i w_i (x^\top A_i x) = x^\top\Big( \sum_{i=1}^m \alpha_i w_i A_i \Big) x. \end{equation} We are interested in investigating the sharpness of this inequality. In particular, we pose the following question: \par
\medskip\emph{When is the upper bound in \eqref{eq:PowerMean-Quad} best possible in the class of quadratic forms?} \medskip\par
Here, ``best possible'' means that every quadratic majorant of the left-hand side also dominates the right-hand side for some $\alpha\in\mathcal A$. When such sharpness holds, it enables the reformulation of nonlinear inequalities involving the power means of quadratic forms, such as \begin{equation*} \begin{aligned} & \forall x : \quad\;\min\big\{ x^\top A_1 x, \dots, x^\top A_m x\big\} && \!\!\!\!\!\!\! \le x^\top B x, && \;\;\text{(case $p=-\infty$)}\\[3pt] & \forall x : \quad\left( \frac{w_1}{x^\top A_1 x}+ \dots+\frac{w_m}{x^\top A_m x} \right)^{\!-1} &&\!\!\!\!\!\!\! \le x^\top B x, && \;\;\text{(case $p=-1$)} \\[5pt] & \forall x : \quad\;\big( x^\top A_1 x\big)^{w_1}\cdots\big( x^\top A_m x\big)^{w_m} &&\!\!\!\!\!\!\! \le x^\top B x, && \;\;\text{(case $p=0$)} \end{aligned} \end{equation*} into \textit{equivalent} LMI conditions of the form \begin{equation*} \exists\alpha\in\mathcal{A} : \quad\sum_{i=1}^m \alpha_i w_i A_i \preceq B, \end{equation*} which is particularly relevant in optimization and control-theoretic contexts. \par
Equivalences of this kind are naturally viewed as lossless $\mathcal{S}$-procedure statements: a pointwise inequality over all $x$ is certified by an LMI with auxiliary multipliers. Our main theorem shows that this certificate is lossless whenever the relevant lower envelope of the augmented quadratic range is convex. \par
\smallskip\par
The paper makes four principal \textit{contributions}: \begin{enumerate} \item It establishes a unified quadratic-majorant theorem for the full concave power-mean family $p\in[-\infty,1)$, including finite sums of such means. \item It identifies lower-envelope convexity as a natural sufficient condition for losslessness, weaker than convexity of the full augmented joint range. \item It proves unconditional sharpness in the two-term case for every concave weighted power mean, yielding a Peter-Paul-type majorant whose $p=-\infty$ endpoint recovers Yuan's lemma. \item It derives exact convex reformulations for the associated maximization problems on the unit sphere and applies the geometric-mean case to a finite-horizon optimal control problem. \end{enumerate} \par
\par
\par
The remainder of the paper is organized as follows. Section~2 positions the results within the literature on the $\mathcal{S}$-procedure, hidden convexity, and geometric-mean optimization. Section~3 develops the lower-envelope geometry and gives automatic and checkable convexity conditions. Section~4 proves the sharp majorant theorem and derives its two-term and variational consequences. Section~5 applies the geometric-mean formulation to finite-horizon optimal control, and Section~6 concludes with open questions. \par
\par
\section{Related Work} \label{sec:related} \par
The problem studied in this paper lies at the confluence of three research traditions: lossless $\mathcal{S}$-procedures and Yuan-type alternatives, hidden convexity and rank-one exactness for quadratic maps, and operator and geometric means. \par
\smallskip\noindent\textit{$\mathcal{S}$-procedure and Yuan-type alternatives.} The main result can be read as a lossless nonlinear $\mathcal{S}$-procedure for homogeneous functions of quadratic forms. Classical $\mathcal{S}$-procedure results certify global quadratic implications by LMIs with multipliers; see the survey of P\'olik and Terlaky~\cite{polik2007survey} and the original work of Yakubovich~\cite{yakubovich1971s}. Related nonquadratic extensions include the U-lemma of Yang, Wang, and Xia~\cite{yang2022toward}, which unifies the $\mathcal{S}$-lemma with the convex Farkas lemma by coupling a pair of quadratic functions with additional convex constraints in a separate variable. The closest endpoint is Yuan's lemma~\cite{yuan1990subproblem}: for two quadratic forms, a pointwise max/min alternative is equivalent to the positive semidefiniteness of a suitable convex combination. The eigenvalue form of this alternative was developed by Mart\'inez-Legaz and Seeger~\cite{martinezlegaz1994yuan}. Extensions beyond two forms generally require additional rank, definiteness, or exactness assumptions; see, for example,~\cite{haeser2017extension,song2023calabi,ai2024tightness}. Our theorem replaces the minimum by arbitrary concave weighted power means, and the multiplier set is the dual power-mean level set coming from the scalar variational formula. \par
\smallskip\noindent\textit{Hidden convexity and rank-one exactness.} The losslessness mechanism belongs to the hidden-convexity theory of quadratic maps; see, for example, Ben-Tal and Teboulle~\cite{bental1996hidden}. Dines' theorem~\cite{dines1941mapping} gives convexity of the joint range of two real homogeneous quadratic forms, while Polyak~\cite{polyak1998convexity} and later work study convexity of higher-dimensional quadratic transformations and its consequences for optimization and control. In this paper the relevant condition is weaker than full convexity of the augmented joint range: only the lower envelope $V_{A,B}$, which assigns to prescribed values of the forms $x^\top A_i x$ the least attainable value of $x^\top Bx$, needs to be convex. The proof of the unconditional two-term case uses semidefinite lifting and the Barvinok-Pataki rank bound~\cite{barvinok1995problems,pataki1998rank} to recover a rank-one optimizer; related rank-one decomposition arguments appear in Sturm and Zhang~\cite{sturm2003cones}. Recent work on exact aggregations of quadratic inequalities~\cite{blekherman2024aggregations} is close in spirit, but our aggregating multipliers are dictated by the dual geometry of power means. \par
\smallskip\noindent\textit{Geometric means: majorants versus minorants.} For the geometric mean, the same scaled arithmetic majorants appear in the SDP relaxation of Yuan and Parrilo~\cite{yuan2021semidefinite} for maximizing products of nonnegative quadratic forms on the sphere, for which they establish hardness and approximation guarantees. Relative to their results, we identify a natural structural condition under which this relaxation is exact, and extend the exact certificate to all powers $p<1$, to unequal weights, and to grouped sums. The two-term geometric case is also complementary to the operator-geometric-mean theory of Pusz-Woronowicz~\cite{pusz1975functional} and Kubo-Ando~\cite{kubo1980means}; see also Bhatia~\cite{bhatia2007positive}. For $A,B\succ0$, the operator geometric mean satisfies $x^\top(A\#B)x\le\sqrt{(x^\top A x)(x^\top B x)}$, and $A\#B$ is a maximal quadratic minorant, although in general no largest quadratic minorant exists. Corollary~\ref{cor:PP-Majorant} addresses the opposite side: sharp quadratic majorants. \par
\smallskip\noindent The distinguishing feature of the present framework is the combination of power-mean dual multiplier geometry with lower-envelope exactness; neither the classical alternatives, whose multipliers have different geometric origins, nor operator-mean theory yields the sharp quadratic majorants obtained here. \par
\section{Lower-Envelope Convexity} \label{sec:lower-envelope} \par
This section develops the lower-envelope convexity condition used in the main results of Section~\ref{sec:main}. The lower envelope records the least value of $x^\top Bx$ compatible with prescribed values of the forms $x^\top A_i x$, and its convexity will later provide the key passage from pointwise information about these quadratic forms to the corresponding matrix-level certificate. \par
\par
Given a tuple of matrices $A = (A_1, \dots, A_m)$ with $A_i \in\mathbb{S}^n$, we define the (unconstrained) joint numerical range as \begin{equation*} W(A) \vcentcolon= \left\{ \big(x^\top A_1 x, \dots, x^\top A_m x\big) \in\mathbb{R}^m \mid x \in\mathbb{R}^n \right\}. \end{equation*} For any $B \in\mathbb{S}^n_{\scriptscriptstyle\succeq0}$, the range $W(A,B) \subseteq\mathbb{R}^{m+1}$ is defined analogously. The associated lower envelope $V_{A,B}:\mathbb{R}^m \to\mathbb{R} \cup\{ + \infty\}$ is defined by \begin{equation*} V_{A,B}(a) \vcentcolon= \inf\left\{x^\top B x \;\middle|\; \big(x^\top A_1 x, \dots, x^\top A_m x\big) = a\right\} \end{equation*} with the convention that $\inf\varnothing= +\infty$. \par
Geometrically, $W(A)$ is the effective domain of $V_{A,B}$, while $V_{A,B}(a)$ records the lower edge of the fiber of $W(A,B)$ above $a$. Thus $W(A,B)$ lies in the epigraph of $V_{A,B}$, but it need not fill that epigraph and may be non-convex even when the epigraph is convex. \par
\begin{proposition} \label{prop:hierarchy} Under the definitions above, the following implications hold: \begin{equation*} W(A,B) \text{ is convex } \Rightarrow\; V_{A,B} \text{ is convex } \Rightarrow\; W(A) \text{ is convex}. \end{equation*} The converse implications do not hold in general. \end{proposition} \textit{Proof.} Convexity of $W(A,B)$, applied to arbitrarily near-minimal points in two fibers, gives the convexity inequality for $V_{A,B}$. Moreover, $W(A)$ is the effective domain of $V_{A,B}$, and the effective domain of a convex function is convex. Examples later in this section show that neither implication is reversible. \qed\par
For $m \le2$, Dines' theorem~\cite{dines1941mapping} guarantees that $W(A)$ is always convex. However, for $m=2$, the augmented range $W(A,B) \subseteq\mathbb{R}^3$ may be non-convex unless additional conditions are met (see~\cite{polyak1998convexity}). As the following proposition states, despite this potential non-convexity of the joint range, its lower envelope remains convex. \begin{proposition} \label{prop:convexity-m2} For $m=2$, the lower envelope $V_{A,B}$ is always convex. \end{proposition} \textit{Proof.} Consider the semidefinite relaxation of the lower envelope obtained by lifting $x \in\mathbb{R}^n$ to $X \in\mathbb{S}^n$, defined by \begin{equation*} \overline{V}_{A,B}(a) \vcentcolon= \inf\left\{ \operatorname{tr}[BX] \;\middle|\; \big(\!\operatorname{tr}[A_1 X], \,\operatorname{tr}[A_2 X]\big) = a, \; X \succeq0 \right\}. \end{equation*} Since $\overline{V}_{A,B}$ is the value function of a semidefinite programming problem, it is convex. The Barvinok-Pataki bound~\cite{barvinok1995problems, pataki1998rank} guarantees that if an optimal solution exists, there is one with rank $r$ satisfying $r(r+1)/2 \le m$. For $m = 2$, this implies $r \le1$. Thus, whenever the infimum is attained, the relaxation admits a rank-one optimal solution $X = x x^\top$, implying $V_{A,B}(a) = \overline{V}_{A,B}(a)$. If the infimum is not attained, replace $B$ by $B+\varepsilon I$. For every $\varepsilon>0$, the lifted objective is coercive on the positive semidefinite cone, so the infimum is attained and the relaxation is exact. As $\varepsilon\downarrow0$, $V_{A,B+\varepsilon I}$ decreases pointwise to $V_{A,B}$ (exchange the two infima). Each $V_{A,B+\varepsilon I} = \overline{V}_{A,B+\varepsilon I}$ is convex, and pointwise limits of convex functions are convex. Therefore, $V_{A,B}$ is convex. \qed\par
\par
In the main results, convexity of the lower envelope is the key structural assumption used to establish losslessness. It is automatic for $m=2$ by Proposition~\ref{prop:convexity-m2}; for larger families, stronger but easier-to-check conditions may be useful. By Proposition~\ref{prop:hierarchy}, convexity of $W(A,B)$ is sufficient, and pairwise commutativity gives a simple algebraic condition. \begin{proposition} \label{prop:commuting} If the matrices in the tuple $(A,B)$ pairwise commute, then the joint numerical range $W(A,B)$ is convex, and therefore $V_{A,B}$ is convex. \end{proposition} \textit{Proof.} Pairwise commutativity gives a common orthonormal eigenbasis. In that basis, $W(A,B)$ is the conic hull of the vectors of corresponding diagonal entries, and is therefore convex. \qed\par
The following two examples demonstrate that the convexity of $V_{A,B}$ is strictly weaker than the convexity of $W(A,B)$, and hence also weaker than commutativity. They also illustrate two simple mechanisms by which lower-envelope convexity may arise. \par
\begin{example}[An effectively two-form example] \label{ex:two-form} Consider the case with $m=3$ and $n=2$. Let $B=I\in\mathbb{S}^2$ and let \begin{equation*} A_1 = \begin{bmatrix} 1 & 0 \\ 0 & 0 \end{bmatrix}, \quad A_2 = \begin{bmatrix} 1 & 1 \\ 1 & 1 \end{bmatrix}, \quad A_3 = \begin{bmatrix} 2 & 1 \\ 1 & 1 \end{bmatrix}. \end{equation*} These matrices are positive semidefinite and do not pairwise commute. However, $A_3=A_1+A_2$. Hence this is, in effect, a two-form example: \begin{equation*} V_{A,I}(a_1,a_2,a_3) = \begin{cases} V_{(A_1,A_2),I} \; (a_1,a_2), & a_3=a_1+a_2,\\ +\infty, & \text{otherwise}. \end{cases} \end{equation*} Therefore $V_{A,I}$ is convex by Proposition~\ref{prop:convexity-m2}. On the other hand, $W(A,I)$ is not convex. Indeed, the points $(1,1,2,1)$ and $(0,1,1,1)$ both belong to $W(A,I)$, but their midpoint $\left(\frac12,1,\frac32,1\right)$ does not. The same mechanism appears in Example~\ref{ex:with-graphs} in Section~\ref{sec:main}. \end{example} \par
Unlike the preceding construction, the three matrices in the following example are linearly independent, so lower-envelope convexity cannot be reduced to the two-form case. \par
\begin{example}[A rank-one sign-pattern example] \label{ex:sign-pattern} Consider the case with $m=n=3$. Let $B=I\in\mathbb{S}^3$, $A_i=u_i u_i^\top$, where $u_1=(1,0,0)^\top$, $u_2=(1,1,0)^\top$, $u_3=(1,1,1)^\top$. Equivalently, \begin{equation*} A_1= \begin{bmatrix} 1&0&0\\ 0&0&0\\ 0&0&0 \end{bmatrix},\quad A_2= \begin{bmatrix} 1&1&0\\ 1&1&0\\ 0&0&0 \end{bmatrix},\quad A_3= \begin{bmatrix} 1&1&1\\ 1&1&1\\ 1&1&1 \end{bmatrix}. \end{equation*} These matrices are positive semidefinite and do not commute. Let $U$ be the matrix with rows $u_i^\top$. Since $U$ is nonsingular, writing $y=Ux$ gives $a_i=(u_i^\top x)^2=y_i^2$, $\|x\|^2=y^\top H y$, where \begin{equation*} H=(UU^\top)^{-1} = \begin{bmatrix} 2&-1&0\\ -1&2&-1\\ 0&-1&1 \end{bmatrix}. \end{equation*} Therefore, for $a\in\mathbb{R}^3_{\scriptscriptstyle\ge0}$, $V_{A,I}(a)=\min_{\sigma\in\{\pm1\}^3}(\sigma\circ\sqrt{a})^\top H(\sigma\circ\sqrt{a})$. Since $H$ has nonpositive off-diagonal entries, the minimum is attained when the signs are equal. Thus \begin{equation*} V_{A,I}(a) = 2a_1+2a_2+a_3 -2\sqrt{a_1a_2} -2\sqrt{a_2a_3}, \qquad a\in\mathbb{R}^3_{\scriptscriptstyle\ge0}, \end{equation*} and $V_{A,I}(a)=+\infty$ outside $\mathbb{R}^3_{\ge0}$. This function is convex, since $-\sqrt{st}$ is convex on $\mathbb{R}^2_{\scriptscriptstyle\ge0}$. Nevertheless, $W(A,I)$ is not convex. For instance, the points $(1,1,1,1)$ and $(0,0,1,1)$ both belong to $W(A,I)$, but their midpoint $\left(\frac12,\frac12,1,1\right)$ does not. Indeed, if $y_i^2=(1/2,1/2,1)$, then the possible values of $y^\top H y$ are $3\pm1\pm\sqrt{2}$, none of which equals $1$. Thus lower-envelope convexity may arise from the sign structure of the quadratic representation even when neither commutativity nor reduction to two effective forms is available. \end{example} \par
The preceding example relies on a single sign pattern that minimizes all cross terms simultaneously. The next example shows that when no such sign pattern exists, lower-envelope convexity can fail even when the forms are rank one and $W(A)$ is convex. \par
\par
\begin{example}[A frustrated sign-pattern example] \label{ex:frustrated} In Example \ref{ex:sign-pattern}, keep $B=I$, $u_1$, $u_2$, and $A_i=u_iu_i^\top$, but replace $u_3$ by $(0,1,1)^\top$. Then $W(A)=\mathbb{R}^3_{\scriptscriptstyle\ge0}$ is convex and $V_{A,I}(a)=\min_{\sigma\in\{\pm1\}^3}(\sigma\circ\sqrt{a})^\top H(\sigma\circ\sqrt{a})$ on $\mathbb{R}^3_{\scriptscriptstyle\ge0}$, now with \begin{equation*} H=(UU^\top)^{-1} = \begin{bmatrix} 3&-2&1\\ -2&2&-1\\ 1&-1&1 \end{bmatrix}. \end{equation*} This time the off-diagonal signs of $H$ are frustrated: minimizing all cross terms would require $\sigma_1\sigma_2=\sigma_2\sigma_3=1$ together with $\sigma_1\sigma_3=-1$, which is impossible. No sign choice is optimal throughout, and convexity is lost: direct enumeration gives $V_{A,I}(0,1,1)=1$, $V_{A,I}(1,1,1)=2$, and $V_{A,I}(2,1,1)=11-6\sqrt{2}$, so that \begin{equation*} 2\,V_{A,I}(1,1,1) \;>\; V_{A,I}(0,1,1)+V_{A,I}(2,1,1). \end{equation*} Thus $V_{A,I}$ is non-convex although $W(A)$ is convex, so the second implication in Proposition~\ref{prop:hierarchy} cannot be reversed. In particular, lower-envelope convexity may fail already for three rank-one positive semidefinite forms. \end{example} \par
The following property provides a simple way to generate larger examples from smaller ones. Together with the constructions of Examples \ref{ex:two-form} and \ref{ex:sign-pattern}, it gives useful sufficient mechanisms for lower-envelope convexity. \par
\begin{proposition} \label{prop:direct-sums} Let $A_i=A_i^{(1)}\oplus\cdots\oplus A_i^{(N)}$ and $B=B^{(1)}\oplus\cdots\oplus B^{(N)}$. If every $V_{A^{(k)},B^{(k)}}$ is convex, then the full lower envelope $V_{A,B}$ is convex. \end{proposition} \textit{Proof.} Decomposing $x$ into its orthogonal blocks shows that $V_{A,B}$ is the infimal convolution of the block lower envelopes. Since infimal convolution preserves convexity, the claim follows.\qed\par
The results above provide several readily verifiable mechanisms for lower-envelope convexity, but they are not intended as a characterization. In the main results below, lower-envelope convexity is imposed directly. \par
\section{Main Results: Sharp Quadratic Majorants} \label{sec:main} \par
Throughout this section, we use the notation and conventions for power means introduced in Section~\ref{sec:intro}. \par
\subsection{Grouped Power-Mean Majorants} \par
The theorem below treats disjoint groups of quadratic forms, each with its own power and weights. Sufficiency follows from the scalar variational formula, while necessity uses a minimax argument and lower-envelope convexity. \par
\par
\begin{theorem}[Sharpness of Power-Mean Majorants] \label{thm:main} Let $\mathcal{I}$ be a finite index set partitioned into disjoint groups $\mathcal{I}_1, \dots, \mathcal{I}_K$. For each $k \in\{1, \dots, K\}$, let $p_k \in[-\infty, 1)$ and let weights $w^k \vcentcolon= (w_i)_{i \in\mathcal{I}_k}$ be given such that $w_i > 0$ and $\sum_{i \in\mathcal{I}_k} w_i = 1$. Let $A = (A_i)_{i \in\mathcal{I}}$ with $ A_i \in\mathbb{S}^{n}_{\scriptscriptstyle\succeq0}$, $A_i \ne0$ and let $B \in\mathbb{S}^{n}_{\scriptscriptstyle\succeq0}$. Denote $q_k = p_k(p_k-1)^{-1}$ and define $\mathcal{A} \vcentcolon= \prod_{k=1}^K \mathcal{A}_k$, where: \begin{equation*} \begin{cases} \mathcal{A}_k \vcentcolon= \Bigl\{ \alpha\in\mathbb{R}_{\scriptscriptstyle\ge0}^{|\mathcal{I}_k|} \Bigm| \underset{\;q_k, \, w^k}{\operatorname{Mean}} \,[ \alpha]=1 \Bigr\}, & \text{if } \, p_k < 0, \\[0.1cm] \mathcal{A}_k \vcentcolon= \Bigl\{ \alpha\in\mathbb{R}_{\scriptscriptstyle> 0}^{|\mathcal{I}_k|} \Bigm| \underset{\;q_k, \, w^k}{\operatorname{Mean}} \,[ \alpha]=1 \Bigr\}, & \text{if } \, 0 \le p_k < 1. \end{cases} \end{equation*} Consider the following statements: \begin{enumerate} \item[$(i)$] For all $x \in\mathbb{R}^n$, \begin{equation*} \sum_{k=1}^K \, \underset{\,p_k, \, w^k}{\operatorname{Mean}} \Big[ \big( x^\top A_i x \big)_{i \in\mathcal{I}_k}\Big] \le x^\top B x. \end{equation*} \par
\item[$(ii)$] There exists $\alpha\in\mathcal{A}$ such that \begin{equation*} \sum_{i \in\mathcal{I}} \alpha_i w_i A_i \preceq B. \end{equation*} \end{enumerate} Then: \begin{itemize} \item[$\bullet$] The implication (ii) $\Rightarrow$ (i) always holds. \item[$\bullet$] If the lower envelope $V_{A,B}$ is convex, the equivalence (i) $\Leftrightarrow$ (ii) holds. \end{itemize} \end{theorem} \par
\textit{Proof.} Recall the variational characterization \eqref{eq:Var-Mean} of a single weighted power mean. Applying it group-wise, we obtain the following scalar equality, which will be used in the sequel: \begin{equation} \label{eq:variational-sum-wgm} \sum_{k=1}^K \underset{\, p_k, \,w^k}{\operatorname{Mean}} \big[(y_i)_{i \in\mathcal{I}_k}\big] = \inf_{\alpha\in\mathcal{A}} \sum_{i \in\mathcal{I}} \alpha_i w_i y_i. \end{equation} \par
Let $\mathcal{D} \vcentcolon= \left\{ X \in\mathbb{S}^n_{\scriptscriptstyle\succeq0} \mid\operatorname{tr}\big[X\big] = 1 \right\}$ be the set of density matrices. To make the proof structure transparent, we introduce an intermediate statement: \begin{enumerate} \item[$(i^*)$] For all $X \in\mathcal{D}$, \begin{equation*} \sum_{k=1}^K \, \underset{\,p_k, \, w^k}{\operatorname{Mean}} \Big[ \big( \operatorname{tr} \!\big[ A_i X \big] \big)_{i \in\mathcal{I}_k}\Big] \le\operatorname{tr} \big[ B X\big]. \end{equation*} \end{enumerate} We now prove the theorem by showing three facts: \begin{enumerate} \item Unconditionally, $(ii) \Rightarrow(i)$. \item Unconditionally, $(i^*) \Leftrightarrow(ii)$. \item If $V_{A,B}$ is convex, then $(i) \Rightarrow(i^*)$. \end{enumerate} \par
\textit{Proof of $(ii) \Rightarrow(i)$.} Given $(ii)$, for any $x \in\mathbb{R}^n$ we have \begin{equation*} \begin{aligned} \sum_{k=1}^K \underset{\, p_k, \,w^k}{\operatorname{Mean}} \Big[ \big( x^\top A_i x \big)_{i \in\mathcal{I}_k}\Big] & = \inf_{\alpha\in\mathcal{A}} \sum_{i \in\mathcal{I}} \alpha_i w_i \big( x^\top A_i x \big) && \text{(by \eqref{eq:variational-sum-wgm} with $y_i=x^\top A_i x$)} \\ & \le x^\top B x, && \text{(by $(ii)$)} \end{aligned} \end{equation*} therefore $(i)$ holds. \par
\textit{Proof of $(i^*) \Leftrightarrow(ii)$.} We introduce a change of variables to convexify the optimization domain. Let $\mathcal{B} \vcentcolon= \prod_{k=1}^K \mathcal{B}_k$, where: \begin{equation} \label{eq:B-change} \begin{cases} \mathcal{B}_k = \big\{ \beta\in\mathbb{R}^{|\mathcal{I}_k|} \mid\sum_{i \in\mathcal{I}_k} w_i \beta_i = 0 \big\}, & \text{if } p_k = 0, \\ \mathcal{B}_k = \big\{ \beta\in\mathbb{R}_{\scriptscriptstyle\ge0}^{|\mathcal{I}_k|} \mid\sum_{i \in\mathcal{I}_k} \beta_i = 1 \big\}, & \text{if } \, p_k < 0, \\ \mathcal{B}_k = \big\{ \beta\in\mathbb{R}_{\scriptscriptstyle> 0}^{|\mathcal{I}_k|} \mid\sum_{i \in\mathcal{I}_k} \beta_i = 1 \big\}, & \text{if } \, 0 < p_k < 1. \\ \end{cases} \end{equation} Consider the change of variables $\alpha= \varphi(\beta)$ defined component-wise by \begin{equation} \label{eq:beta-change} \varphi(\beta_i) \vcentcolon= \begin{cases} e^{\beta_i}, & \text{if } i \in\mathcal{I}_k \text{ with }p_k = 0, \\ (w_i^{-1} \beta_i)^{1/q_k}, & \text{if } i \in\mathcal{I}_k \text{ with } p_k \ne0. \end{cases} \end{equation} Define the functions $\Phi: \mathcal{A} \times\mathcal{D} \to\mathbb{R}$ and $\Phi_e: \mathcal{B} \times\mathcal{D} \to\mathbb{R}$ as \begin{equation*} \Phi(\alpha, X) := \operatorname{tr}\left[ \left( \sum_{i \in\mathcal{I}} \alpha_i w_i A_i - B \right) X \right]\!, \quad\Phi_{e}(\beta, X) := \Phi(\varphi(\beta), X). \end{equation*} Observe the following properties: \begin{enumerate} \item For fixed $\beta$, the map $X \mapsto\Phi_e(\beta, X)$ is linear (and thus concave and continuous) on the compact convex set $\mathcal{D}$. \item For fixed $X$, the map $\beta\mapsto\Phi_e(\beta,X)$ is convex on the convex set $\mathcal B$. \end{enumerate} Therefore, Sion's Minimax Theorem applies~\cite{sion1958general}: \begin{equation*} \sup_{X \in\mathcal{D}} \inf_{\beta\in\mathcal{B}} \Phi_e(\beta, X) = \inf_{\beta\in\mathcal{B}} \sup_{X \in\mathcal{D}} \Phi_e(\beta, X). \end{equation*} The map $\varphi$ defines a bijection from $\mathcal{B}$ onto $\mathcal{A}$. Consequently, the optimization over $\mathcal{B}$ is equivalent to optimizing over $\mathcal{A}$. Translating the result back to the original variables yields: \begin{equation} \label{eq:sion-minimax} \sup_{X \in\mathcal{D}} \inf_{\alpha\in\mathcal{A}} \Phi(\alpha, X) = \inf_{\alpha\in\mathcal{A}} \sup_{X \in\mathcal{D}} \Phi(\alpha, X). \end{equation} Finally, recall that for all $P \in\mathbb{S}^n$, it holds that \begin{equation} \label{eq:density-to-unit-vector} \sup_{X \in\mathcal{D}} \operatorname{tr}[PX] = \sup_{\|x\|=1} \operatorname{tr}[Px x^\top] = \lambda_{\max} (P). \end{equation} Putting the above observations together, we obtain: \begin{equation*} \begin{aligned} (i^*) \quad& \Leftrightarrow&& \sup_{X \in\mathcal{D}} \, \inf_{\alpha\in\mathcal{A}} \, \Phi(\alpha,X) \le0 && \text{(by \eqref{eq:variational-sum-wgm} with $y_i = \operatorname{tr} \big[ A_i X\big]$)} \\ & \Leftrightarrow&& \inf_{\alpha\in\mathcal{A}} \, \sup_{X \in\mathcal{D}} \, \Phi(\alpha,X) \le0 && \text{(by \eqref{eq:sion-minimax})} \\ & \Leftrightarrow&& \inf_{\alpha\in\mathcal{A}} \, \sup_{\|x\|=1} \Phi(\alpha,x x^\top) \le0 && \text{(by \eqref{eq:density-to-unit-vector})} \\ & \Leftrightarrow&& \inf_{\alpha\in\mathcal{A}} \lambda_{\max}\Big( \sum_{i \in\mathcal{I}} \alpha_i w_i A_i - B \Big) \le0, \end{aligned} \end{equation*} which is equivalent to $(ii)$, provided the infimum is attained. \par
Finally, we confirm the attainment of the infimum. The function $\alpha\mapsto\lambda_{\max} \left( \sum_{i \in\mathcal{I}} \alpha_i w_i A_i - B \right)$ is continuous and the domain is a product of group-wise constraint sets. For groups where $p_k < 0$, the corresponding set is compact. For groups where $0 < p_k < 1$, the constraint with $q_k<0$ implies that every $\alpha_i$ is bounded away from zero. For groups where $p_k=0$, if any component $\alpha_i \to0$, then the constraint forces some other component $\alpha_{i'} \to\infty$. Finally, since $A_i \succeq0$ and $A_i \neq0$, the objective function is coercive, so no component can tend to infinity along a minimizing sequence. Hence every such sequence has a convergent subsequence with limit in the domain. Therefore, the infimum is attained at some $\alpha\in\mathcal{A}$, completing the proof of $(i^*) \Leftrightarrow(ii)$. \par
\textit{Proof of $(i) \Rightarrow(i^*)$.} For any $X \in\mathcal{D}$ a decomposition exists in the form \begin{equation} \label{eq:density-spectral-decomp} X = \sum_{j=1}^n \lambda_j v_j v_j^\top, \quad\lambda_j \ge0, \quad\sum_{j=1}^n \lambda_j = 1, \quad\|v_j\| =1. \end{equation} Define $a: \mathcal{D} \to\mathbb{R}^{|\mathcal{I}|}$ by $a_i(X) \vcentcolon= \operatorname{tr} [A_i X]$, $a(X) \vcentcolon= \big( a_i (X) \big)_{i \in\mathcal{I}}$ and note that \begin{equation*} a(X) = \sum_{j=1}^n \lambda_j \, a(v_j v_j^\top). \end{equation*} By $(i)$, for any $x \in\mathbb{R}^n$ satisfying $x^\top A_i x = a_i (X)$ for all $i \in\mathcal{I}$ we have \begin{equation*} \sum_{k=1}^K \underset{\,p_k, \, w^k}{\operatorname{Mean}} \Big[ \big( \! \operatorname{tr}[A_i X] \big)_{i \in\mathcal{I}_k}\Big] = \sum_{k=1}^K \underset{\,p_k, \, w^k}{\operatorname{Mean}} \Big[ \big( x^\top A_i x \big)_{i \in\mathcal{I}_k}\Big] \le x^\top B x. \end{equation*} Taking the infimum over such $x$ yields \begin{equation} \label{eq:ineq-to-lower-boundary} \sum_{k=1}^K \underset{\,p_k, \, w^k}{\operatorname{Mean}} \Big[ \big( \! \operatorname{tr}[A_i X] \big)_{i \in\mathcal{I}_k}\! \Big] \le\inf\left\{x^\top B x \;\middle|\; x^\top A_i x = a_i(X) \right\} \! = \! V_{A,B}\big(a(X)\big). \end{equation} Now suppose that $V_{A,B}$ is convex. Then, by Jensen's inequality, \begin{equation} \label{eq:jensen-lower-boundary} V_{A,B}\big(a(X)\big) = V_{A,B} \Big(\sum_{j=1}^n \lambda_j \, a(v_j v_j^\top) \Big) \le\sum_{j=1}^n \lambda_j V_{A,B} \big(a (v_j v_j^\top)\big). \end{equation} By definition of $V_{A,B}$, we have \begin{equation} \label{eq:vk-upper-bound} V_{A,B} \big(a (v_j v_j^\top)\big) = \inf\left\{x^\top B x \;\middle|\; x^\top A_i x = v_j^\top A_i v_j\right\} \le v_j^\top B v_j. \end{equation} Combining the above, we conclude that for every $X\in\mathcal{D}$, \begin{equation*} \begin{aligned} \sum_{k=1}^K \underset{\,p_k, \, w^k}{\operatorname{Mean}} \Big[ \big( \! \operatorname{tr}[A_i X] \big)_{i \in\mathcal{I}_k}\Big] & \le V_{A,B}\big(a(X)\big) && \text{(by \eqref{eq:ineq-to-lower-boundary})} \\ &\le\sum_{j=1}^n \lambda_j V_{A,B} \big(a(v_j v_j^\top)\big) && \text{(by \eqref{eq:jensen-lower-boundary})} \\ & \le\sum_{j=1}^n \lambda_j v_j^\top B v_j && \text{(by \eqref{eq:vk-upper-bound})} \\ & = \operatorname{tr} \big[ B \sum_{j=1}^n \lambda_j v_j v_j^\top\big] \; && \text{(by trace properties)} \\ & = \operatorname{tr} \big[ B X\big], && \text{(by \eqref{eq:density-spectral-decomp})} \end{aligned} \end{equation*} which yields $(i^*)$, as desired. \qed\par
\textit{Remark.} The equivalence in Theorem~\ref{thm:main} may be viewed as a lossless $\mathcal{S}$-procedure for grouped weighted power means. The scalar supporting-majorant formula determines the multiplier geometry, while lower-envelope convexity ensures that lifting from vectors to density matrices introduces no gap. \par
\par
\subsection{The Two-Term Peter-Paul Majorant} \label{section:Peter-Paul} For two terms, the convexity hypothesis in Theorem~\ref{thm:main} is automatic: Proposition~\ref{prop:convexity-m2} shows that the relevant lower envelope is always convex. The resulting certificate is therefore unconditional, and the two-variable multiplier set can be parametrized by a single scalar, giving the following Peter-Paul-type majorant. \par
\begin{corollary}[Sharpness of Peter-Paul Majorant] \label{cor:PP-Majorant} Let $A, B, C \in\mathbb{S}^n_{\scriptscriptstyle\succeq0}$, $A, B \ne0$. Let $w=(a,b)$ with $a, b > 0$ and $a+b=1$. Let $p \in[-\infty, 1)$ and $q = p(p-1)^{-1}$. For $p\ne0$, define \begin{equation*} I_p \vcentcolon= \begin{cases} [0,1], & p<0,\\ (0,1), & 0<p<1. \end{cases} \end{equation*} Consider the following statements: \begin{enumerate} \item[$(i)$] For all $x \in\mathbb{R}^n$, \begin{equation*} \underset{p, \, w}{\operatorname{Mean}} \big[ x^\top A x, \, x^\top B x \big] \le x^\top C x. \end{equation*} \item[$(ii)$] There exists $s > 0$ such that \begin{equation*} a s^{b} A + b s^{-a} B \preceq C. \end{equation*} \item[$(iii)$] There exists $s \in I_p$ such that \begin{equation*} a^{\frac{1}{p}} s^{\frac{1}{q}} A + b^{\frac{1}{p}} (1-s)^{\frac{1}{q}} B \preceq C. \end{equation*} \end{enumerate} Then: \begin{itemize} \item[$\bullet$] If $p = 0$, the equivalence $(i) \Leftrightarrow(ii)$ holds (geometric mean case). \item[$\bullet$] If $p\ne0$, the equivalence $(i) \Leftrightarrow(iii)$ holds (other power mean cases). \end{itemize} \end{corollary} \par
\textit{Proof.} This follows immediately from Theorem \ref{thm:main} applied with weights $w=(a,b)$ by reparametrizing the dual constraint set $\mathcal{A}$. If $p=0$, the constraint is $\alpha_1^a \alpha_2^b = 1$, which is parametrized by setting $\alpha_1 = s^b, \alpha_2 = s^{-a}$ for $s > 0$. If $p \ne0$, the constraint is $a \alpha_1^q + b \alpha_2^q = 1$, which is parametrized by setting $\alpha_1 = (s/a)^{1/q}$ and $\alpha_2 = ((1-s)/b)^{1/q}$ for $s \in I_p$. In both cases, substituting these into the LMI condition $a \alpha_1 A + b \alpha_2 B \preceq C$ yields the stated forms. The equivalence holds without additional assumptions because the variational lower envelope is always convex for $m=2$ (Proposition \ref{prop:convexity-m2}). \qed\par
\textit{Remark.} At $p=-\infty$, Corollary~\ref{cor:PP-Majorant} reduces to Yuan's lemma after writing $A_1=C-A$ and $A_2=C-B$: \begin{equation*} \max\{x^\top A_1 x,\, x^\top A_2 x\}\ge0 \ \ \forall x \quad\Leftrightarrow\quad\exists\, \theta\in[0,1]:\ \theta A_1+(1-\theta)A_2\succeq0. \end{equation*} At $p=0$, statement $(ii)$ gives the weighted Young family of quadratic majorants. For equal weights, it is the quadratic-form version of the classical Peter-Paul inequality $uv\leq\frac{\varepsilon}{2}u^2+\frac{1}{2\varepsilon}v^2$. This majorant result complements the operator geometric mean, which addresses the minorant side. \par
\subsection{Extrema on the Unit Sphere} \par
The certificate also yields extremal formulas on the unit sphere. \par
\begin{corollary}[Variational Characterization] \label{cor:variational} Let the notation be as in Theorem \ref{thm:main} and denote $S(x) \vcentcolon=\sum_{k=1}^K \operatorname{Mean}_{\,p_k, \, w^k} \big[ ( x^\top A_i x)_{i \in\mathcal{I}_k}\big]$. Then: \begin{enumerate} \item[$(i)$] Unconditionally, \begin{equation} \inf_{\|x\|=1} S(x) = \inf_{\alpha\in\mathcal{A}} \lambda_{\min} \Big( \sum_{i \in\mathcal{I}} \alpha_i \, w_i A_i \Big). \end{equation} \item[$(ii)$] If the lower envelope $V_{A,I}$ is convex, then \begin{equation} \sup_{\|x\|=1} S(x) = \inf_{\alpha\in\mathcal{A}} \lambda_{\max} \Big( \sum_{i \in\mathcal{I}} \alpha_i \, w_i A_i \Big). \end{equation} \end{enumerate} \end{corollary} \par
\textit{Proof.} For part $(i)$, taking the joint infimum over $x$ and $\alpha$ gives \begin{equation*} \begin{aligned} \inf_{\|x\|=1} S(x) = \inf_{\|x\|=1} \inf_{\alpha\in\mathcal{A}} \sum_{i \in\mathcal{I}} \alpha_i w_i \big( x^\top A_i x \big) = \inf_{\alpha\in\mathcal{A}} \inf_{\|x\|=1} x^\top\Big( \sum_{i \in\mathcal{I}} \alpha_i w_i A_i \Big) x, \end{aligned} \end{equation*} which yields the claim. For part $(ii)$, note that, unconditionally, for all $\alpha\in\mathcal{A}$, \begin{equation*} \begin{aligned} s \vcentcolon= \sup_{\|x\|=1} S(x) \le\sup_{\|x\|=1} \sum_{i \in\mathcal{I}} \alpha_i w_i \big( x^\top A_i x \big) = \lambda_{\max} \Big( \sum_{i \in\mathcal{I}} \alpha_i w_i A_i \Big), \end{aligned} \end{equation*} therefore $s \le\inf_{\alpha\in\mathcal{A}} \lambda_{\max} \big( \sum\alpha_i w_i A_i \big)$. Conversely, apply Theorem \ref{thm:main} with $B = s I$. Then statement $(i)$ of Theorem \ref{thm:main} holds and, since $V_{A,I}$ is convex, statement $(ii)$ of that theorem follows: there exists $\alpha\in\mathcal{A}$ such that $\sum\alpha_i w_i A_i \preceq s I$. Consequently, $\inf_{\alpha\in\mathcal{A}} \lambda_{\max} \big( \sum\alpha_i w_i A_i \big) \le s$. Combining the two inequalities establishes the equality. \qed\par
\textit{Remark.} For $p=0$ and a single group, part $(ii)$ is the exact version, under lower-envelope convexity, of the SDP relaxation studied by Yuan and Parrilo~\cite{yuan2021semidefinite} for products of positive semidefinite quadratic forms on the sphere. Part $(i)$ is unconditional because minimization allows the two infima to be exchanged, whereas the maximization identity in $(ii)$ is precisely where losslessness of the semidefinite relaxation is needed. The change of variables~\eqref{eq:B-change}, \eqref{eq:beta-change} then turns the dual side into the convex eigenvalue program used below. \par
\begin{example} \label{ex:with-graphs} To illustrate the variational characterization for multiple groups, we consider an objective combining a weighted harmonic mean, a geometric mean, and a minimum. Let $A_1, \dots, A_8 \in\mathbb{S}^2_{\scriptscriptstyle\succeq0}$ be defined as \begin{equation*} A_1 = \begin{bmatrix} 8 & 2 \\ 2 & 5 \end{bmatrix}, \quad A_2 = \begin{bmatrix} 7 & 1 \\ 1 & 4 \end{bmatrix}, \quad A_3 = \begin{bmatrix} 6 & 0 \\ 0 & 3 \end{bmatrix}, \quad A_4 = \begin{bmatrix} 5 & -1 \\ -1 & 2 \end{bmatrix} \end{equation*} \begin{equation*} A_5 = \begin{bmatrix} 2 & 2 \\ 2 & 2 \end{bmatrix}, \quad A_6 = \begin{bmatrix} 4 & -2 \\ -2 & 1 \end{bmatrix}, \quad A_7 = \begin{bmatrix} 6 & 6 \\ 6 & 6 \end{bmatrix}, \quad A_8 = \begin{bmatrix} 21 & -9 \\ -9 & 6 \end{bmatrix}. \end{equation*} Subject to $\|x\|=1$, we seek to maximize the function $f : \mathbb{R}^2 \to\mathbb{R}$: \begin{equation*} \begin{aligned} f(x) = \left( \frac{1/6}{x^\top A_1 x} + \frac{1/3}{x^\top A_2 x} + \frac{1/2}{x^\top A_3 x}\right)^{-1} & + \sqrt[3]{\big( x^\top A_4 x\big)\big( x^\top A_5 x\big)\big( x^\top A_6 x\big)} \\[3pt] & + \min\big\{ x^\top A_7 x, \, x^\top A_8 x\big\}. \end{aligned} \end{equation*} \par
The matrices $A_i$ do not all commute pairwise, so simultaneous diagonalization is unavailable. Nevertheless, the forms span only a two-dimensional space. Indeed, every $A_i$ is a linear combination of $A_3$ and $A_5$. The lower envelope $V_{A,I}$ therefore coincides, up to a linear embedding, with $V_{(A_3,A_5),I}$, which is convex by Proposition~\ref{prop:convexity-m2}. Direct maximization is possible and provides a useful numerical check, while the equivalent eigenvalue minimization illustrates the formulation that remains applicable in higher-dimensional grouped problems. \par
\par
\begin{figure}[t] \centering\begin{minipage}[t]{0.48\textwidth} \centering\includegraphics[width=\linewidth]{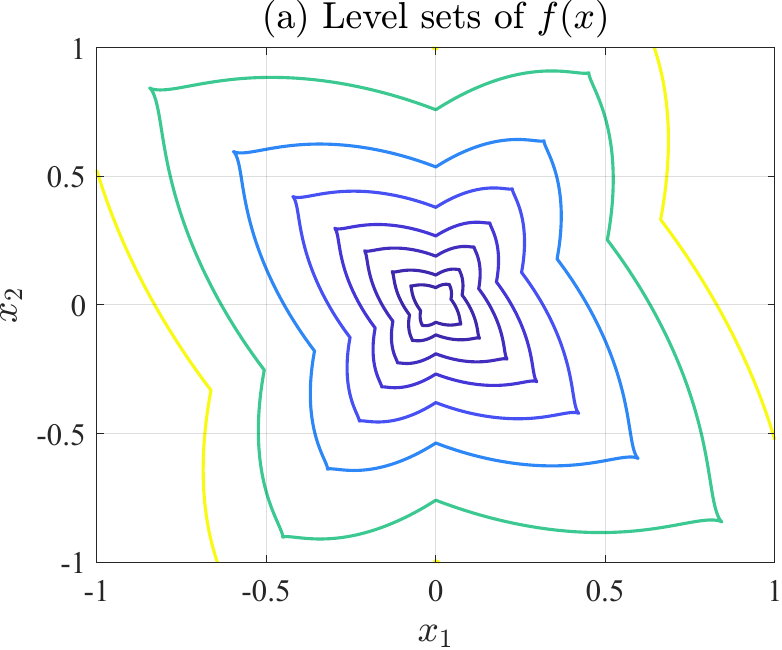} \end{minipage}\hfill\begin{minipage}[t]{0.48\textwidth} \centering\includegraphics[width=\linewidth]{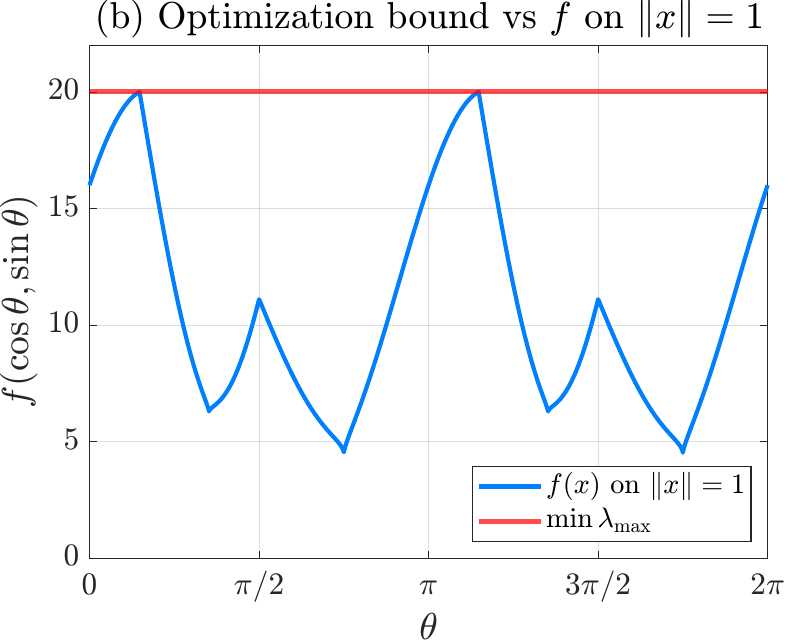} \end{minipage} \caption{Visualization of Example \ref{ex:with-graphs}. (a) Level sets of $f$. (b) The value of $f$ on the unit circle $\|x\|=1$ alongside the tight upper bound obtained from the optimization problem.} \label{fig:asymmetric-side-by-side} \end{figure} \par
By Corollary~\ref{cor:variational}, the maximum equals the optimal value of the following optimization problem: \begin{equation*} \begin{aligned} & \underset{\alpha\in\mathbb{R}^8_{\ge0}}{\text{minimize}} && \lambda_{\max} \bigg( \sum_{i=1}^8\alpha_i w_i A_i \bigg) \\ & \text{subject to} && \big( w_1 \sqrt{\alpha_1}+ w_2 \sqrt{\alpha_2}+ w_3 \sqrt{\alpha_3} \big)^2 = 1, \\[0.1cm] &&& \alpha_4^{w_4} \alpha_5^{w_5} \alpha_6^{w_6} = 1, \quad w_7 \alpha_7+ w_8 \alpha_8 = 1, \end{aligned} \end{equation*} where $w = (1/6, \,1/3, \,1/2, \, 1/3, \, 1/3, \, 1/3, \, 1/2, \, 1/2)$. In order to convexify this problem, we apply the change of variables introduced in \eqref{eq:B-change}, \eqref{eq:beta-change} to obtain the following optimization problem: \begin{equation*} \begin{aligned} & \underset{\beta\in\mathbb{R}^8}{\text{minimize}} && \lambda_{\max} \Bigg( \sum_{i=1}^3 \frac{\beta_i^2}{w_i} A_i \;+\; \sum_{i=4}^6 \frac{e^{\beta_i}}{3} A_i \;+\; \sum_{i=7}^8 \beta_i A_i \Bigg) \\ & \text{subject to} && \beta_1+\beta_2+\beta_3 = 1,\quad\beta_1,\beta_2,\beta_3 \ge0, \\ &&& \frac{\beta_4}{3}+\frac{\beta_5}{3}+\frac{\beta_6}{3} = 0, \quad\beta_7+\beta_8 = 1,\quad\beta_7,\beta_8 \ge0. \end{aligned} \end{equation*} Each scalar coefficient in the matrix-valued map is convex and each $A_i$ is positive semidefinite, so the matrix-valued objective map is convex in the Loewner order. Since $\lambda_{\max}$ is convex and monotone in that order, the objective is convex, while all constraints are linear. Thus the original non-convex maximization is represented exactly by a convex eigenvalue program. \par
Parametrizing the unit circle as $x=(\cos\theta,\sin\theta)$ and solving the transformed convex program give the same optimal value: $\max_{\|x\|=1} f(x) \approx20.0103$. In Figure~\ref{fig:asymmetric-side-by-side}, the left panel displays the geometry of the nonquadratic objective, while the right panel shows contact between $f$ and the optimal quadratic upper bound at the maximizing directions. \end{example} \par
\section{An Application to Optimal Control: Coincident Excitation} \label{sec:control} \par
Many systems possess a trigger that responds not to any single observable but to the product of several. Because the response is multiplicative, it becomes appreciable only in the rare event that several quantities peak at the same instant, and a shortfall in one of them cannot be bought back by a surplus in another. The pattern recurs across domains. In photochemistry, the rate of a multiphoton process is proportional to the product of the illuminating intensities, each intensity itself a quadratic function of the field~\cite{denk1990,lakowicz1996,helmchen2005}. In cellular signaling, a coincidence detector such as a molecular gate is engaged only when several pathways are driven simultaneously, its activation growing with the product of the corresponding drives~\cite{seeburg1995,philip2010}. In structural dynamics, the energy exchanged across a weak nonlinear coupling between vibrational modes (as in an internal resonance) grows with a product of the energies of the participating modes, so that initiating the exchange requires all of them to be excited together~\cite{kerschen2009,asadi2021}. In every instance the driving quantities are quadratic measures of the state (energies, powers, or intensities), their product is the natural figure of merit for the trigger, and the actuation that we are free to design carries a limited budget of energy, whether a fuel reserve or a safety limit that guards against damage to tissue or hardware. This provides a potential application of the preceding analysis to control theory through the following formalization. \par
\smallskip\par
\textit{Problem formulation.} Consider a controllable linear system \begin{equation} \label{eq:lti} \dot{x}(t) = A x(t) + B u(t), \quad x(0) = 0, \quad x(t) \in\mathbb{R}^n, \quad u(t) \in\mathbb{R}^r, \end{equation} started from rest and actuated by an input whose energy is limited by \begin{equation} \label{eq:energy} \int_0^T u(t)^\top R \, u(t) \, \mathrm{d}t \le1, \end{equation} where $R \in\mathbb{S}^r_{\scriptscriptstyle\succ0}$ weights the channels of $u$ according to their cost. Given matrices $Q_1, \dots, Q_N \in\mathbb{S}^n_{\scriptscriptstyle\succeq0}$, $Q_i \ne0$, together with weights $w_i > 0$ obeying $\sum_{i=1}^N w_i = 1$, we seek the input that maximizes the peak value of the corresponding weighted geometric mean (equivalently, the $p=0$ weighted power mean studied above) of these quadratic quantities over the time interval $[0,T]$: \begin{equation} \label{eq:control-objective} J^\star\vcentcolon= \sup_{u} \; \sup_{0 \le t \le T} \; \prod_{i=1}^N \big( x(t)^\top Q_i x(t) \big)^{w_i}, \end{equation} the outer supremum ranging over all inputs satisfying \eqref{eq:energy}. (Positive exponents that do not sum to one are brought to this normalization by setting $w_i = p_i / \sum_j p_j$ and raising the value to the power $\sum_j p_j$, an increasing transformation that leaves both the optimal input and the analysis unchanged.) Since the weights sum to one, the objective is homogeneous of degree two in the state, matching the scale fixed by the budget \eqref{eq:energy}. \par
\smallskip\par
\textit{Solution.} Although problem \eqref{eq:control-objective} is nonconvex and of infinite dimension, it reduces to one of finite dimension. Let \begin{equation*} P_T = \int_0^T e^{A \tau} B R^{-1} B^\top e^{A^\top\tau} \, \mathrm{d}\tau\end{equation*} be the finite-horizon controllability Gramian of \eqref{eq:lti} under the energy weighting $R$. For each $t>0$, the least energy required to steer the system from rest to a state $x_f$ at time $t$ is $x_f^\top P_t^{-1}x_f$; see, for example,~\cite{boyd1994linear}. Hence the states reachable at time $t$ under the unit budget form the ellipsoid defined by $x_f^\top P_t^{-1}x_f\leq1$. Since $P_t\preceq P_T$ for $0\leq t\leq T$, every such ellipsoid is contained in the terminal ellipsoid. Conversely, every point of the terminal ellipsoid is reachable at time $T$ with energy at most one. Therefore, the union of states reachable over $0\leq t\leq T$ is exactly the terminal controllability ellipsoid. \par
Using homogeneity of degree two to place the maximum on the boundary, write $x_f=P_T^{1/2}z$ and define $\widetilde Q_i\vcentcolon=P_T^{1/2}Q_iP_T^{1/2}$. The problem then becomes \begin{equation} \label{eq:Jstar-ellipsoid} J^\star= \sup_{\|z\|=1} \; \prod_{i=1}^N \big( z^\top\widetilde Q_i z \big)^{w_i}. \end{equation} Define $\mathcal{A}=\{\alpha\in\mathbb{R}^N_{\scriptscriptstyle>0}\,|\,\prod_{i=1}^N\alpha_i^{w_i}=1\}$ and $\mathcal{B}=\{\beta\in\mathbb{R}^N\,|\,\sum_{i=1}^Nw_i\beta_i=0\}$. If the lower envelope $V_{\widetilde Q,I}$ is convex, Corollary~\ref{cor:variational} gives the following exact eigenvalue formulation: \begin{equation} \label{eq:Jstar-dual} J^\star= \inf_{\alpha\in\mathcal{A}} \lambda_{\max}\Bigl(\sum_{i=1}^N\alpha_iw_i\widetilde Q_i\Bigr) = \inf_{\beta\in\mathcal{B}} \lambda_{\max}\Bigl(\sum_{i=1}^Nw_ie^{\beta_i}\widetilde Q_i\Bigr). \end{equation} In the absence of a structural condition ensuring exactness, maximization of a product of positive semidefinite quadratic forms is NP-hard~\cite{yuan2021semidefinite}. \par
For $N=2$, Proposition~\ref{prop:convexity-m2} makes $V_{\widetilde Q,I}$ convex automatically, so \eqref{eq:Jstar-dual} is exact without any additional assumption. This applies, for example, to the above-mentioned two-color two-photon excitation, where two quadratic intensity factors must peak together. For $N=1$, $Q_1=I$, one has $\widetilde Q_1=P_T$, the admissible set reduces to $\alpha_1=1$, and \eqref{eq:Jstar-dual} gives $J^\star=\lambda_{\max}(P_T)$, the largest squared state norm reachable under the unit-energy constraint, as expected. \par
\par
\par
The reduction is also constructive. Let $\alpha^\star$ solve the first problem in \eqref{eq:Jstar-dual}, and set $M^\star=\sum_i\alpha_i^\star w_i\widetilde Q_i$. If the largest eigenvalue of $M^\star$ is simple, its unit eigenvector $z^\star$ maximizes \eqref{eq:Jstar-ellipsoid}. If the largest eigenvalue is multiple, one selects within the top eigenspace a unit vector satisfying $\alpha_i^\star z^{\star\top}\widetilde Q_i z^\star=J^\star$ for every $i$; exactness guarantees that such a choice exists, although it need not be unique. The corresponding terminal state is $x_f^\star=P_T^{1/2}z^\star$. \par
The input of least energy that steers the state from rest to $x_f^\star$ at time $T$ is \begin{equation} \label{eq:optimal-input} u^\star(\tau) = R^{-1} B^\top e^{A^\top(T - \tau)} P_T^{-1} x_f^\star, \qquad0 \le\tau\le T, \end{equation} and it has energy $x_f^{\star\top} P_T^{-1} x_f^\star= 1$. It therefore satisfies the budget and delivers the objective value $J^\star$ at time $T$. \par
\smallskip\par
The application therefore has three distinct layers. The reduction to the terminal controllability ellipsoid is unconditional; equality with the convex eigenvalue program requires lower-envelope convexity; and recovery from a dual solution is immediate when the top eigenvalue is simple, while in the multiple-eigenvalue case one selects an appropriate vector from the top eigenspace. The power-mean certificate separates these steps and turns the original infinite-dimensional control problem into a finite-dimensional exact formulation under the stated convexity condition. \par
\par
\par
\section{Conclusions} \par
This paper studied the sharpness of scalar variational majorants of weighted power means when the scalar arguments are replaced by positive semidefinite quadratic forms. The central condition is convexity of the lower envelope of the joint numerical range: weaker than convexity of the augmented range, yet strong enough to render the semidefinite lifting exact, so that a pointwise inequality involving sums of power means becomes equivalent to a single linear matrix inequality in the dual power-mean multipliers. For two terms, rank-one exactness makes the condition automatic, and the Peter-Paul majorant is sharp unconditionally. \par
When the condition holds, the variational corollary turns nonconvex maximization on the sphere into convex eigenvalue minimization. The optimal-control application uses this constructively, reducing a coincident-excitation problem under an energy budget to a finite-dimensional convex program whose solution yields a minimum-energy optimal input. \par
Several questions remain open. Most importantly, can lower-envelope convexity be characterized by tractable conditions beyond the mechanisms of Section~\ref{sec:lower-envelope}? Is it genuinely necessary for exactness, or can the certificate remain lossless under weaker geometric assumptions? And when exactness fails, can the relaxation gap be bounded in terms of a quantitative measure of nonconvexity of the lower envelope? Further directions include indefinite forms, constrained domains, and other scalar variational identities whose exactness may be governed by related hidden-convexity mechanisms. \par
\section*{Declarations} \par
No funding was received for conducting this study. The author has no relevant financial or non-financial interests to disclose. No external datasets were used or generated in this study. \par
 \par
\par
\end{document}